\documentclass[12pt]{amsart}
\usepackage{graphicx} 
\usepackage{amsmath,graphics}
\usepackage{tikz}
\usepackage{amsfonts,amssymb}
\usepackage{xypic}
\usepackage{comment}
\usepackage{color}
\usepackage{hyperref}
\usepackage{cleveref}
\usepackage{thmtools}
\usepackage{thm-restate}
\usepackage{mathtools}
\usepackage{enumitem}
\makeatletter\let\@wraptoccontribs\wraptoccontribs\makeatother

\title[Finiteness properties of algebraic fibers of group extensions]{Finiteness properties of algebraic fibers of group extensions}

\author{Dessislava H. Kochloukova}
\address{Department of Mathematics, State University of Campinas (UNICAMP), S\~ao Paulo, Brasil}
\email{desi@unicamp.br}
\author{Stefano Vidussi}
\address{Department of Mathematics, University of California, Riverside, CA 92521, USA} \email{svidussi@ucr.edu}
\date{\today}

\theoremstyle{plain}
\newtheorem*{theorem*}{Theorem}
\newtheorem*{corollary*} {Corollary}
\newtheorem*{proposition*}{Proposition}
\newtheorem*{conjecture*}{Conjecture}
\newtheorem{theorem}{Theorem}[section]
\newtheorem{corollary}{Corollary}[section]
\newtheorem{lemma}{Lemma}[section]
\newtheorem{proposition}{Proposition}[section]
\newtheorem{conjecture}{Conjecture}[section]

\theoremstyle{remark}
\newtheorem*{remark}{Remark}

\newtheorem{definition}{Definition}[section]
\newtheorem*{ack}{Acknowledgements}
\theoremstyle{definition}

\def\op{\operatorname}

\def\G{\Gamma}

 \def\Q{\mathbb{Q}}  \def\Z{\mathbb{Z}} \def\R{\mathbb{R}}  
    
\def\lr{\longrightarrow}
    \def\bp{\begin{pmatrix}}
 \def\ep{\end{pmatrix}} \def\bn{\begin{enumerate}} 
   \def\en{\end{enumerate}}
\def\ba{\begin{array}} \def\ea{\end{array}}

  \def\Ker{\operatorname{Ker}}
\def\S{\Sigma}
\def\ker{\op{Ker}}\def\be{\begin{equation}} \def\ee{\end{equation}}
\def\An{\op{Aut}(F_n)} \def\A2{\op{Aut}(F_2)}  
\def\On{\op{Out}(F_n)} \def\O2{\op{Out}(F_2)}
 \def\SA2{\op{Aut^{+}}(F_2)} 
\def\IA2{\op{Inn}(F_2)} 
 \def\SO2{\op{Out^{+}}(F_2)}
 \def\Sy4{\op{S}_4}

 \def\dim{\mbox{dim\hspace*{1pt}}}

\def\op{\operatorname}

\newcommand\wh{\widehat}
\newcommand\out{\operatorname{Out}}
\newcommand\aut{\operatorname{Aut}}
\newcommand\Inn{\operatorname{Inn}}
\newcommand\ssm{\smallsetminus}
\newcommand\hP{\widehat{\Pi}}
\newcommand\hookra{\hookrightarrow}
\begin{document}

\contrib[with an appendix by]{Marco Boggi}

\maketitle

\begin{abstract} This survey describes some recent work, by the authors and others, on the existence of algebraic fibrations of group extensions, as well as the finiteness properties of their algebraic fibers, in the realm of both abstract and pro-$p$ groups. We also discuss some applications of these results to (higher) coherence.
    
\end{abstract}

\section{Introduction}

Let $G$ be a finitely generated group, and let $\phi \colon G \to \Z$ be an epimorphism. We say that $\phi$ is an {\em algebraic fibration} if $\ker{\phi} \unlhd G$ is finitely generated (f.g. henceforth), in which case we refer to $\ker{\phi}$ as an {\em algebraic fiber}. The study of algebraic fibrations of groups is a storied and, at the same time, current topic of research in group theory. It dates back, at least morally, to the study of {\em fibered knots}, classical knots whose exterior $N(K) = S^3 \setminus \nu K$ is a locally-trivial surface bundle over a circle, developed (among others) by Lee Neuwirth, Elvira Rapaport Strasser and John Stallings in the early 1960's. The surface in object is compact, with a single boundary component, hence its fundamental group is $F_{2g}$, the free group on $2g$ generators, where $g$ is the genus of the knot. (Some of the facts casually stated here are actually nontrivial.) The long exact sequence of a fibration reduces, in this case, to the fact that the abelianization map $\phi \colon \pi_1(N(K)) \to \Z$ has kernel $F_{2g}$. More generally, given a compact, irreducible (i.e. such that all $2$-spheres bound $3$--disks) oriented $3$-manifold $N$ that fibers over the circle (which entails that the boundary, if nonempty, is a disjoint union of tori), we have an epimorphism $\phi \colon \pi_1(N) \to \Z$ with kernel a surface group, i.e. an algebraic fibration. In the $3$--manifold case, a remarkable fact, due to Stallings, is that an algebraic fibration on a $3$--manifold group $\pi_1(N)$ actually ``integrates" to a locally trivial surface bundle structure on $N$, and in particular the group $\ker{\phi}$ is the fundamental group of a finite complex. This fact is simultaneously central and peculiar to the study of $3$--manifolds. 

The realm of this survey is the class of groups that can be written as extensions of f.g. groups, i.e. appear in a short exact sequence of the type
\begin{equation} \label{eq:ext} 1 \lr K \lr G \stackrel{f}{\lr} \G \lr 1 \end{equation} where $K$ and $\G$ are themselves finitely generated. With a slight abuse of notation, we will refer to groups of this type {\em tout court} as extensions. Our goal is to determine conditions under which these extensions admit a algebraic fibrations and discuss, in a sense, ``what can go wrong" (compared with the $3$--manifold case) in terms of regularity of the fiber. In particular, we will find instances of a variety of finiteness properties that the kernels of algebraic fibrations (even in geometric situations) can have. 

This survey is structured as follows. In Section \ref{sec:fin} we will define the finiteness conditions we are interested in, and their interplay with the BNSR invariants. In Section \ref{sec:ext} we will introduce the class of groups in the title, namely groups that are extensions of a f.g. group by a f.g. group. Section \ref{sec:main} contains the main results implying or obstructing finiteness of the algebraic fibers, as well as some applications. Section \ref{sec:prof} discusses pro-$p$ groups, their presentations, homology and finiteness.
Section \ref{sec:ab} relates the theories of finitely presented pro-$p$ and (abstract) metabelian groups. Section \ref{sec:extpro} discusses the pro-$p$ version of some of the main results in Section \ref{sec:main}, and Section \ref{sec:coh} is devoted to some applications of these results to coherence. Further results on coherence follow by applying the results contained in the Appendix written by Marco Boggi.
\begin{ack} It is a pleasure to thank Stefan Friedl, Rob Kropholler and Genevieve Walsh for many conversations on the topics of this survey. The second author would also like to thank Luca Di Cerbo and Laurentiu Maxim for organizing a stimulating Session on the Singer-Hopf Conjecture in Geometry and Topology at the AMS Sectional Meeting of March 2023 in Atlanta, GA, in which many of the results contained in this survey were presented. We thank the referee for the numerous suggestions. The first named author was partially supported by Bolsa de produtividade em pesquisa CNPq 305457/2021-7 and Projeto tem\'atico FAPESP 18/23690-6.
\end{ack}

\section{Finiteness properties of groups and BNSR invariants} \label{sec:fin} 
We start by introducing finiteness properties of groups, that can be thought of as strengthening of the condition of finite generation. These come in two broad flavors, that can be thought of as {\em homotopical} or {\em homological}. (We refer \cite{Br94,Br10} for a thorough discussion of these notions, and the many related facts that will be implicitly used in what follows.)

\begin{definition} \label{def:fin} Let $G$ be a group and let $n$ be a nonnegative integer.
\begin{itemize}
\item we say that $G$ {\em has type $F_n$} if it admits as classifying space a CW-complex $K(G,1)$, which we can assume without loss of generality to have a single vertex, whose $n$-skeleton is finite; furthermore we say it has type $F$ if the $K(G,1)$ is a finite complex;
\item we say $G$ {\em has type $FP_n(R)$}, where $R$ is a commutative ring with unity, if $R$, as a trivial $R[G]$-module, admits a resolution by projective $R[G]$-modules
\[ \ldots \lr P_i \lr P_{i-1} \ldots \lr P_0 \lr R \lr 0, \]
with the first $n$-terms being f.g. as $R[G]$-modules.
\end{itemize} 
\end{definition}

We usually write $FP_n$ to mean type $FP_n(\Z)$. For a group $G$, finiteness type $F_1$ (or $FP_1(R)$) is equivalent to being f.g., and finiteness type $F_2$ is equivalent to finite presentation (f.p. henceforth); thanks to work of Bestvina--Brady there exist however groups of type $FP_2$ (classically referred to as {\em almost finitely presented}), and even $FP_{\infty}$, that fail to be $F_2$, see \cite{BB97}. In more generality, finiteness type $F_n$ entails $FP_n(R)$ for all $R$, but the converse fails. In Section \ref{sec:prof} we will extend this definition to the case of pro-$p$ groups.

As outlined before, we will mostly be concerned with the finiteness properties of algebraic fibers, which are by definition kernels of an epimorphism from a group $G$ to $\Z$, hence in particular $[G,G] \unlhd \ker{\phi} \unlhd G$. Conveniently, the study of the finiteness properties of such subgroups fits in the broader context of the Bieri--Neumann--Strebel--Renz (BNSR) invariants of the ``ambient" group $G$. Much like finiteness properties, these exist in both homotopical and homological form. These were defined originally in \cite{BNS87} (for the $n=1$ case in which they are more often referred to as BNS invariants), and then extended to all $n$ in \cite{R88,BR88}.

In order to define them, consider the projectivization of the first cohomology group $H^{1}(G;\R)$, namely the {\em character sphere} \[ S(G) :=(H^{1}(G;\R) \setminus \{0\})/\R^{+}; \] elements $\chi \in S(G)$ are referred to as {\em characters}, and a character is {\em discrete} if one of its representatives in $H^{1}(G;\R)$ is contained in the lattice $H^{1}(G;\Z) \subset H^{1}(G;\R)$. Stated differently, there exists a representative of $\chi$ which is given by an epimorphism $\phi \colon G \to \Z$.

It is convenient to define the $0$-th homotopical and homological invariant $\S^{0}(G) = \S^{0}(G,R) = S(G)$. To define the $n$-th (homotopical) BNSR invariant $\S^{n}(G) \subset S(G)$, assume now that $G$ has type $F_n$, and let $\G_{n}$ be the $n$-skeleton of the universal cover of some $K(G,1)$. Choosing $K(G,1)$ to have a single vertex, $\G_0$ is in bijection with $G$ and $\G_1$ is the 
Cayley graph for some finite generating set. Given a character $\chi \in S(G)$, let $(\G_n)_{\chi}$ be the subcomplex of $\G_n$ defined by the vertices $v \in \Gamma_{0}$ that satisfy $\chi(v) \geq 0$ (this definition is independent of the choice of representative of $S(G)$ in $H^{1}(G;\R)$).  
\[ \S^{n}(G) = \{\chi \in S(G) | (\G_{n})_{\chi} \text{~ is $(n-1)$-connected for some choice of $\G_n$} \}. \]

To define the $n$-th (homological) BNSR invariant $\S^{n}(G,R) \subset S(G)$, define much as above $G_{\chi}$ to be the submonoid defined by the elements $g \in G$ such that $\chi(g) \geq 0$. We say that $R$ is of type $FP_n$ as a (trivial) $R[G_{\chi}]$ module if it admits a resolution with the first $n$-terms f.g. projective $R[G_{\chi}]$-modules. Then 
\[ \S^{n}(G,R) = \{\chi \in S(G) | R \text{~ is of type $FP_n$ as $R[G_{\chi}]$-module} \}. \] If $G$ is not of type $FP_n(R)$, we put $\S^{n}(G;R) = \emptyset$.

The first BNSR invariant $\S^{1}(G) = \S^{1}(G,R)$ is the same for the homotopical and homological variants, and irrespective of the choice of ring $R$, while the invariants may differ at higher degree. 

The connection between the finiteness properties of algebraic fibers and BNSR invariants is distilled in the following result, whose proof combines work contained in \cite{BR88,R88,R89}.

\begin{proposition} \cite{BR88,R88,R89} Let $G$ be a f.g. group, and let $\chi \in S(G)$ be a discrete character.
If $G$ has type $F_n$ (respectively type $FP_n$), then $\ker \chi \unlhd G$ has type $F_n$ (respectively type $FP_n$ if and only if $\pm \chi \in \S^{n}(G)$ (respectively $\pm \chi \in \S^{n}(G,R)$).
\end{proposition} 

On face value, this proposition embeds the problem of determining the finiteness properties of an algebraic fiber of $G$ into the broader problem of studying the BNSR invariants of $G$: this brings an entire apparatus to the study of the original problem, and in particular offers the advantage that the BNSR invariants of $G$ are {\em open} subsets of the character sphere (endowed with the subspace topology arising as unit sphere $S(G) \subset H^{1}(G;\R) \cong \R^{b_1(G)}$). The added flexibility of working with non-discrete characters is comparable to use cohomology with real coefficients as opposed to integer ones, and this is frequently more than analogy and becomes a guiding principle, as in the development of the Thurston norm as tool to study genus-minimizing surfaces in $3$--manifolds. 

\section{Extensions} \label{sec:ext}
The study of BNSR invariants of groups is tightly related with our ability to understand the properties of the group, and it is rarely straightforward. The scope of this survey is therefore quite narrow, and is directed towards a very specific class of groups, namely f.g. groups $G$ that arise as extensions of f.g. groups. Two classes of groups come to mind in that realm: free-by-free groups (where both $K$ and $\G$ are free groups) and fundamental groups of aspherical manifolds which admit a locally trivial fibration. The latter class notably contain surface-by-surface groups. 

The case where the extension is trivial, i.e. $G = K \times \G$, is the first example that comes to mind. In that case it is well-known that the first BNSR invariant of $G$ can be read directly from those of $K$ and $\G$: precisely we have the relation (more easily phrased in terms of complementary subsets)
\[ \S^{1}(G)^{c} = \bigcup_{0 \leq p \leq 1} \S^{p}(K)^c \ast \S^{1-p}(\G)^c \] and a similar relationship (necessarily) holds for the homological invariants. (Here, characters on $K$ and $\G$ are interpreted as characters on $G$ by pull-back.) Extending this result to $n \geq 2$ is notoriously intricate and partial results, as well as counterexamples to the most obvious generalizations (the ``product formula"), hold, see \cite{BG10}. One could hope, at a minimum, that the BNSR invariants of $G$ are ``affected" by those of $K$ and $\G$, and in fact it is not difficult to show that if a character of $\G$ has a given finiteness property, so does the corresponding pull-back character of $G$.  The situation is more complex regarding characters from $K$ because they may not extend, or they do not canonically, as characters of $G$ (there's no pull-back!). There's one easy exception (which, in a sense, is complementary to the product case), and that is the case where $b_1(G) = b_1(\G)$, so that $S(G) = S(\G)$ (properly, under identification under pull-back). In that case it is not too difficult to see that as we have assumed that $K$ is f.g. then $\S^{1}(G) = \S^{1}(\G)$. (Assuming suitable stronger assumptions on the finiteness properties of $K$, we can deduce a similar equality result for higher BNSR invariants.) 

In fact, this result generalizes to a class of groups that can be written as extensions, and that arise in Complex Geometry, namely K\"ahler groups. 

\begin{definition} A complex manifold $X$ is K\"ahler if it can be endowed with an Hermitian metric such that the associated $2$-form is closed. Fundamental groups of compact K\"ahler manifolds are referred to as {\em K\"ahler groups}.  
\end{definition}
Differently from what happens with complex manifolds (which can have any f.p. group as fundamental group) K\"ahler groups admit significant restrictions. Implicitly, one such restriction is the outcome of the next result, appearing in \cite{De10}, that constraints the form of the first BNS invariant for K\"ahler groups (obviously, with $b_1(G) > 0$). We will formulate it in a way that emphasizes the fact, {\em a priori} non-obvious, that these groups can be thought of as extensions as in Equation (\ref{eq:ext}).
\begin{theorem} \label{thm:delzant} \cite{De10} Let $G$ be a K\"ahler group. Then either 
\begin{itemize}
\item $G$ has finitely generated commutator subgroup, hence is an extension of the form
\[ 1 \lr [G,G] \lr G \lr H_1(G) \lr 1 \]
or
\item there exist a finite collection of hyperbolic orbisurface groups ${\G_i,i \in I}$ and extensions \[ 1 \lr K_i \lr G \stackrel{f_i}{\lr} \G_i \lr 1 \] such that \[ \S^{1}(G) = S(G) \setminus \coprod_{i \in I} f_{i}^{*} S(\G_i) \]
\end{itemize}
\end{theorem}
Some remarks are in order. First, the reader can find the definition of hyperbolic orbisurface group in \cite{Bu11} (which is an excellent survey on the result above), but for our intents it is sufficient to know that they contain as finite index subgroup the fundamental group of an hyperbolic surface, which is known to have empty BNSR invariants. From this it follows that $\S^{1}(\G_i) = \emptyset$, which explains why all characters of $G$ that pull-back from the $\G_i$ are singular. What is not trivial is that all other characters are in $\S^{1}(G)$. Furthermore, unless there exists an $f_{i} \colon G \lr \G_i$ inducing an isomorphism in homology with rational coefficients (in which case we say that $X$, and by extension $G$, has {\em Albanese dimension 1} and $f_i$ can be proven to be essentially unique), then  $\S^1(G)$ is nonempty and, by openness, $G$ admits an algebraic fibration. Lastly one can prove that the subspaces $f_{i}^{*}(H^{1}(\G_{i},\Q)) \subset H^{1}(G;\Q)$ are pairwise orthogonal, see e.g. \cite{Ca03}

A rather informative example of application of Theorem \ref{thm:delzant} comes by looking at Kodaira fibrations. These are K\"ahler manifolds (more precisely, complex projective surfaces) which admit one (or more) structure of surface bundle over a surface in such a way that the fibration map is holomorphic. It follows that their fundamental group $G$ is a surface-by-surface group, and barring the case where these manifolds have Albanese dimension one, we have $\S^1(G)  \neq \emptyset$. This (and obviously much more) is a consequence from Theorem \ref{thm:delzant} and in our understanding a proof not involving the fact that these groups are K\"ahler did not exist. 

\section{Finiteness properties the fibers} \label{sec:main}
In light of Theorem \ref{thm:delzant} one can ask whether there is a proof that $\S^1(G)  \neq \emptyset$ for non-K\"ahler surface-by-surface groups. (For rather deep reasons, namely the Parshin-Arakelov rigidity, surface-by-surface bundles are ``generically" not K\"ahler.) More generally, one can ask for necessary conditions for an extension as in Equation (\ref{eq:ext}) to have nonempty higher BNSR invariants, using as tool only the fact that the group is an extension (as opposed to Theorem \ref{thm:delzant}, whose proof hinges on the K\"ahler assumption). 
The first results in that sense are contained in \cite{KW22,FV23}, which deal with the first BNSR invariant (i.e. $n_0=1$ in the notation below), and these are extended in \cite{KV23}.

\begin{theorem} \label{thm:main} \cite{KW22,FV23,KV23} Let $1 \to K \to G \to \G \to 1$ be a short exact sequence of groups such that $G$ and $K$ are of type $F_{n_0}$ (resp. $FP_{n_0}(R)$) and $b_1(\G) > 0$. Suppose that $\phi \colon K \to \Z$ is an epimorphism with kernel $\ker \phi = N$ of type $F_{n_0-1}$ (resp. $FP_{n_0-1}(R)$) such that  $\phi$  extends to a real homomorphism of $G$. Then there exists an epimorphism $\psi \colon G \to \Z$ that extends $\phi$ such that $\Ker \psi$ is of type $F_{n_0}$ (resp. $FP_{n_0}(R)$). Furthermore if $K$, $G$ and $N$ are of type $F_{\infty}$ (resp. $FP_{\infty}(R)$) then $\Ker \psi$ can be chosen of type $F_{\infty}$ (resp. $FP_{\infty}(R)$).
\end{theorem} 
\begin{remark} Note that when $n_0 = 1$, the existence of a $\phi \colon K \to \Z$ that extends to $G$ is equivalent to the condition that the inclusion-induced homomorphism $H_1(K) \to H_G$ (which, by the Lyndon-Hochschild-Serre sequence factorizes through the quotient map $H_1(K) \to H_1(K)_{\G}$) has image of positive rank, or equivalently $b_1(G) > b_1(\G)$. This condition is referred to by saying that the extension in Equation (\ref{eq:ext}) has excessive homology. As any $\phi \colon K \to \Z$ has kernel of type (at least) $F_0$, if the extension has excessive homology then $G$ admits an epimorphism $\psi \colon G \to \Z$ with kernel of type $F_1$, hence $\pm [\psi] \in \S^1(G)$. 
\end{remark}

To give an idea of the proof, start with a presentation of the base group $\G$, thought of as an epimorphism from the free group $F_n$ to $\G$ for suitable $n$. 
This epimorphism induces a diagram
\begin{equation} \label{eq:epi} \xymatrix{
1\ar[r]&
K \ar[r]\ar[d]_{\cong} &
\Pi \ar[d] \ar[r] & F_n \ar[d] \ar[r]&1\\
 1\ar[r]& K 
 \ar[r]& G \ar[r]  &
\G \ar[r]  &1 } \end{equation}
where all vertical maps are epimorphisms. 

Observe that we can write $\Pi$ as amalgamated product \[ \Pi = \Pi_{1} \ast_K \Pi_2 \ast_K \dots \ast_K \Pi_n \] where each $\Pi_i$ has the form of an HNN extension  $\Pi_i = \langle K,s_i|s_i k s_i^{-1} = f_i(k) \rangle$ for some automorphism $f_i\colon K \to K$. We can then work with the group $\Pi$: the advantage of doing so is that amalgamated products are, in a sense, easier to deal with in terms of study of the BNSR invariants. The condition on excessive homology allows one to establish the existence of algebraic fibrations on $\Pi$, by ``perturbing" any of the fibrations of the $\Pi_i$. (To do so, the openness, and the very use, of the BNSR invariants, is essential.) As long as $b_1(\G) > 0$ at least some of the fibrations descend to $G$, and with some delicate arguments one can control the finiteness properties of the fiber to the extent stated.

We want to discuss an interesting consequence of Theorem \ref{thm:main}, that follows from repeated application of this result to iterated extension. Omitting for sake of simplicity full generality (for which the reader is referred to \cite[Theorem 1.6]{KV23}) we will consider the following set up.
\begin{definition} A {\em normal filtration of length $n$} of a group $G$ is a collection of normal subgroups $G_i \unlhd G$, $0 \leq i \leq n$ such that \[ 1 = G_0 \unlhd G_1 \unlhd \ldots \unlhd G_{n-1} \unlhd G_n = G. \] We say that $G$ is poly-free (respectively poly-surface) if the factor groups $\G_{j+1} := G_{j+1}/G_{j}$, $0 \leq j \leq n-1$ are f.g. free (respectively surface groups).
\end{definition}  
\begin{remark} In literature the term poly-free and poly-surface usually applies also to the case where the filtration is just {\em subnormal}, i.e. each element of the filtration is normal in the successor, but not necessarily in $G$. The result below applies to the more general case, but statements can become quite cumbersome, so we limit ourselves to the class above.
    
\end{remark}

We have the following.
\begin{theorem} \label{thm:poly} \cite{KV23} Let $G$ be a poly-free or a poly-surface group of length $n$, and assume that for all $0 \leq j \leq n-1$ the sequence \begin{equation} \label{eq:excess} 1 \lr G_{j+1}/G_j \lr G/G_{j} \lr G/G_{j+1} \lr 1 \end{equation} has excessive homology. Then for every discrete character $\phi \colon G_1 \to \R$ that extends to a real character on $G$ there exists a discrete character $\psi \colon G \to \R$ extending $\phi$ with kernel of type $F_{n-1}$.
\end{theorem}
The strategy of the proof consists in iteratively identifying a discrete character on $G_{j+1}$ with kernel of increasing regularity, and induct on $0 \leq j \leq n-1$. 

The outcome of Theorems \ref{thm:main} and \ref{thm:poly} is that we can, under the assumption that the sequences in Equation \ref{eq:excess} have excessive homology, identify algebraic fibrations with (higher) finiteness property: for instance, a poly-free group of length $3$ with required excessive homology properties would contain an algebraic fiber of type $F_2$, namely finitely presented. 

One may want to ask if this result is optimal. The quick answer is negative: for instance $G = F_2 \times F_2 \times \Z$ admits an obvious fibration whose fiber has type $F$. Note that, in this case, the Euler characteristic of the group $G$ is zero. In fact, at least for poly-free and poly-surface groups the nonvanishing of the Euler characteristic is the only condition required for optimality of the results above.

This is a consequence of the following result, that originates from \cite[Theorem 7.2(5)]{Lu02}, and which is further discussed in \cite{Fi21,LIMP21}: 
\begin{theorem} \label{thm:lu} \cite{Lu02} Let $G$ be a f.g. group and assume that the $i$-th $L^2$-Betti number $b_{i}^{(2)}(G)$ is non-zero: then the fibers of all algebraic fibrations of $G$ cannot have type $FP_i(\Q)$.
\end{theorem}
This result has relevence in the case of poly-free and poly-surface groups in light of the following known fact, whose proof can be found e.g. in \cite{GN21,KV23}.
\begin{proposition} \label{prop:obs} Let $G$ be a poly-free (respectively poly-surface) group of length $n$. Then $b^{(2)}_{i}(G) =0$ for $0 \leq i \leq n-1$, and \[ b^{(2)}_{n}(G) = (-1)^{n}\chi(G) = (-1)^{n} \prod_{i=1}^{n} \chi (\G_i).\]
\end{proposition} (Note that for poly-surface groups this statement amounts to the fact that an aspherical manifold carrying $G$ as fundamental group satisfies the Singer Conjecture.)

It follows that, unless one of the factor groups is $\Z$ (in the poly-free case) or $\Z^2$ (in the surface case) then an algebraic fiber cannot have type $FP_n(\Q)$ nor, {\em a fortiori}, $F_n$. In particular, the fundamental group $G$ of a surface bundle over a surface where both base and fiber have genus at least $2$ and the first Betti number  of the total space is strictly bigger than that of the base admits an algebraic fibration, so that $\S^{1}(G) \neq \emptyset$, consistently with the case of Kodaira fibrations certified by Theorem \ref{thm:delzant}.

There's a interesting class of groups where the assumption on excessive homology is easily verified, namely the pure braid groups. The pure mapping class group $P(n+1)$ in $n \geq 1$ strand is defined as the kernel of the epimorphism $B(n+1) \to S_{n+1}$ from the braid group in $n+1$ strands to the symmetric group, and geometrically corresponds to braids that do not permute their ends. While the BNSR invariants of the braid group $B(n+1)$ are rather uninteresting ($H_1(B(n+1);\Z) = \Z$ and it is not hard to show that the commutator subgroup has type $F_{\infty}$, hence $\Sigma^{m}(B(n+1)) = S(B(n+1))$ for all $m$) they are much richer for $P(n+1)$. 
For these groups, the first BNSR invariant has been completely elucidated in \cite{KMM15}, and some partial results on higher BNSR invariants appear in \cite{Za17} using Morse theory. We can essentially recast analog results using Theorems \ref{thm:main} and Theorem \ref{thm:lu}:

\begin{theorem} \cite{Za17,KV23} Let $G = P(n+1)$ be the pure braid group on $n+1$ strands; then for all $1 \leq m \leq n-2$, $G$ contains an algebraic fibration with fiber of type $F_m$ but not of type $FP_{m+1}$.
\end{theorem}

The key points of our proof are the following. First, interpreting $P(n+1)$ as the fundamental group of the ordered configuration space of $(n+1)$ points on the plane, we have as a consequence of the Fadell--Neuwirth fibration (a point--forgetting map) a way to describe $P(n+1)$ iteratively as \begin{equation} \label{eq:fn} 1 \lr F_{q} \lr P(q+1) \lr P(q) \lr 1, \end{equation} which together with the fact that $P(2) = \Z$ allows to interpret $P(n+1)$ as poly-free group of length $n$, with normal filtration given by \[ G_i = \ker(P(n+1) \to P(n+1 - i))\] and factor groups $\G_i = F_{n+1-i}$. Note that this entails that $\chi(P(n)) = 0$. However, we have an isomorphism \[ P(n+1) \cong K(n+1) \times \Z \] where $K(n+1) \cong P(n+1)/Z(P(n+1))$ is isomorphic to the penultimate term $G_{n-1} = \ker(P(n+1) \to P(2))$ of the filtration of $P(n+1)$, see \cite[Seection I.9]{FM12}. The group $K(n+1)$ is then poly-free of length $n-1$, and none of its factor groups is free abelian, so its $(n-1)$--th $L^2$--Betti number does not vanish. Finally, it is well--known that the monodromy of all the extensions in Equation \ref{eq:fn} acts trivially on the homology of the fiber, so homologically $P(n+1)$ (and $K(n+1)$) behaves as a direct product as far as the first homology group is concerned. It's a straightforward consequence of this fact that the excessive homology in the sequences in Equation (\ref{eq:excess}) is nontrivial (actually, maximal) so we can apply Theorem \ref{thm:main} to $K(n+1)$ (hence to $P(n+1)$) and obtain the result desired.

Applications of the machinery of Theorems \ref{thm:poly} and \ref{thm:lu} allows one also to get some information, via purely group--theoretical techniques, on higher BNSR invariants of some K\"ahler groups. We state an instance of this:

\begin{theorem} \cite{KV23,LIP23} There exist (complex projective) iterated Kodaira fibrations of complex dimension $n$ with (K\"ahler) fundamental group $G$ whose BNSR invariants are properly nested in the form \[ \S^{0}(G) \supsetneq \S^{1}(G) \supsetneq \ldots \supsetneq \S^{n-1}(G) \supsetneq \S^{n}(G) \]
\end{theorem}

(Much as in the case of the pure braid groups, all the ``jumps" can be obtained by using algebraic fibrations.) A proof of this theorem using the results of this paper appears in \cite{KV23}, while a proof based on a different approach appears in \cite{LIP23}

It is interesting to compare at this point the statements of Theorems \ref{thm:main} and \ref{thm:poly} to those obtained by Kielak and Fisher in the realm of virtually RFRS group. Recall the following 
\begin{definition} A group $H$ is residually finite rationally solvable (RFRS) if there exists a normal filtration $H_i \unlhd_f H$ such that $\cap_{i} H_i = \{1\}$ and the maps $H_{i} \to H_i/H_{i+1}$ factorize through the maximal free abelian quotient maps.
\end{definition} 
In this realm Kielak (\cite{Kie20} for $n=1$) and then Fisher (\cite{Fi21} for $n > 1$) proved the following:
\begin{theorem} \cite{Kie20,Fi21} \label{thm:kf} Let $H$ be a virtually RFRS group of type $FP_{n}(\Q)$. Then there exist a finite index subgroup of $H$ which admits an algebraic fibration with fiber of type $FP_{n}(\Q)$ if and only if the $L^2$-Betti numbers satisfy $b_{i}(H) = 0$ for $0 \leq i \leq n$
\end{theorem}
The obstruction to higher regularity, of course, is consistent with what is prescribed by Theorem \ref{thm:lu}. On the other hand, one may ask about the overlap between the statements of Theorems \ref{thm:main} or \ref{thm:poly} and that of Theorem \ref{thm:kf} as far as the construction of an algebraic fibration with the regularity prescribed. While a complete picture still eludes us, the overlap in the domain of application of these two types of results is likely to be quite limited. For instance, while direct products of virtually RFRS groups (e.g. free groups) are virtually RFRS, it is possible to prove (for rather non-trivial reasons, see \cite{K-V-W}) that there are free-by-free groups that are not virtually fibered, hence not virtually RFRS. This lends to the expectation that poly-free groups (and more generally nontrivial extensions) likely fail to be virtually RFRS.

One way to use the results discussed here is in the realm of higher coherence of groups. 
Recall that a group $G$ is called {\em coherent} if all its finitely generated subgroups are finitely presented. We say that a group is $m$--coherent if all its subgroups of type $F_m$ are of type $F_{m+1}$. A combination of Theorem \ref{thm:poly} and Theorem \ref{thm:lu} yields the following
\begin{corollary} \label{cor:only} Let $G$ be a poly-free or a poly-surface of length $n$ satisfying the excessive homology conditions of Equation (\ref{eq:excess}) with nonvanishing Euler characteristic: then for all $0 \leq m \leq n-1$, $G$ is not $m$-coherent. 
\end{corollary}

By definition, $m$--coherence descends to subgroups, so to prove that a group is not $m$-coherent we just need to find a poison subgroup that fails this property. Corollary \ref{cor:only} is, in essence, a recipe to build poison groups out of (suitable) poly-free and poly-surface groups without abelian factor groups. This idea, of course, is  not new, and the use of embeddings of $F_2 \times F_2$ (or other direct products) is a simple but useful tool to prove incoherence. 

Unsurprisingly, however, direct products of free (or surface) groups cannot always be used to get optimal results, so the strategy outlined above adds some new results that are out of reach otherwise. An instance is the following application to automorphism groups of free groups.

\begin{theorem} \cite{V23} For $n \geq 3$ there exist an algebraically fibered poly-free group of nonvanishing Euler characteristic of length $2n-2$ (respectively $2n-3$), which embeds in $\An$ (respectively $\On$), such that the sequences in Equation (\ref{eq:excess}) have excessive homology. Therefore $\An$ (respectively $\On$) is not $m$-coherent for $0 \leq m \leq (2n-3)$ (respectively $0 \leq m \leq (2n-4)$).
\end{theorem}

Note that direct products of free groups of ``high enough" length do not embed by \cite{HW20}. Details of the proof, which hinges on the embedding of suitable (non product) poly-free groups as in the statement, due to \cite{BKK02}, can be found in \cite{V23}. 

We finish this discussion on coherence with an observation: it has been conjectured (see \cite{Wi19}) that the assumption on excessive homology, for free-by-free groups, is not necessary, even when iterated over finite index subgroups. Stated otherwise, for these groups, there should be witnesses to incoherence that are not necessarily (not even virtually) algebraic fibers. It is probably not too far fetched to conjecture the following  generalization of Corollary \ref{cor:only}, whose proof however could prove rather challenging:
\begin{conjecture} Let $G$ be a poly-free or a poly-surface of length $n$ with nonvanishing Euler characteristic: then for all $0 \leq m \leq n-1$, $G$ is not $m$-coherent. 
\end{conjecture}

  \section{Profinite groups} \label{sec:prof}

\subsection{Pro-$\mathcal{C}$ groups} \label{sec:pro}
Our main reference for profinite groups is the Ribes-Zalesskii book \cite{R-Z}.
 In what follows we will generalize some of the results previously discussed to pro-$\mathcal{C}$ group, where $\mathcal{C}$ is a class of finite groups. We briefly recall the definition.
\begin{definition} A pro-$\mathcal{C}$ group $G$ is the inverse limit of its finite quotients that are in the class $\mathcal{C}$ i.e.
$$G = {\varprojlim} ~ G/ U,$$ where $G/U \in \mathcal{C}$. If $\mathcal{C}$ is the class of all finite groups we say the group is profinite. If $\mathcal{C}$ is the class of all finite $p$-groups, were $p$ a natural prime number, we say the group is pro-$p$. 
\end{definition}
A pro-$\mathcal{C}$ group is a topological group, where the above subgroups $U$ form a nested set of open neighbourhoods of $1_G$. Alternatively $G$ embeds as a closed subgroup in the direct product $\prod G/ U$, where each $G/ U$ is a topological group with the discrete topology and the  topology of $\prod G/ U$ is the product topology.  Every pro-${\mathcal{C}}$ group  is compact and totally disconnected.

By definition a subset $X$ of a pro-$\mathcal{C}$ group $G$ is a (topological) generating set  if there is no proper closed subgroup of $G$ that contains $X$. For an abstract group we have that every element of the group is a finite word on the set of generators and their inverses but this breaks completely in the category of  profinite groups. But still as in the case of abstract groups the nice pro-$\mathcal{C}$ groups are the (topologically) finitely generated ones.

We note that if $U$ is an open subgroup of a pro-$\mathcal{C}$ group $G$ then we have a decomposition as a disjoint union $G = \cup_{t \in T} U t$, where $T$ is transversal. Then since each $ U t$ is an open subset of $G$ and $G$ is compact, we have that $T$ is finite i.e. $U$ has finite index in $G$.
The converse holds for (topologically) finitely generated  profinite groups i.e. every subgroup of finite index in a  (topologically) finitely generated  profinite group is open.
This is a deep result by Nikolov and Segal,  see \cite{N-S1}, \cite{N-S2}, and it means  that the topology of every (topologically) finitely generated  profinite group is uniquely determined by its underlying algebraic structure.  

From now on we restrict to pro-$p$ groups $G$. They form a special class of profinite groups and are much better behaved compared to the general class of profinite groups. A pro-$p$ group is procyclic if it is (topologically) generated by one element. The only infinite procyclic pro-$p$ group is the group of the $p$-adic numbers $\mathbb{Z}_p$ that is the inverse limit ${\varprojlim}  ~ \mathbb{Z} / p^i \mathbb{Z}$. 

We discuss briefly the Frattini subgroup of a pro-$p$ group, for more details the reader can check in \cite[section 2.8]{R-Z}.
By definition the Frattini subgroup $\Phi(G)$ of a pro-$p$ $G$ is the intersection of all maximal closed subgroups of $G$. The Frattini subgroup  $\Phi(G)$ is the pro-$p$ subgroup of $G$ (topologically) generated by the set $$\{ g^p \ | \ g \in G \} \cup \{ [g_1, g_2] = g_1^{-1} g_2^{-1} g_1 g_2 \ | \ g_1, g_2 \in G \}.$$ It turns out that a subset $X$ of $G$  is a (topological) generating set of $G$ if and only if the image of $X$ in $G / \Phi(G)$ is a (topological) generating set. In particular $G$ is (topologically) finitely generated if and only if $G/  \Phi(G)$ is finite, and in this case the minimal number of generators of $G$ is the dimension of  $G/  \Phi(G)$ as a vector space over the field with $p$ elements.

\subsection{Finitely presented pro-$p$ groups}

For an abstract group $H$ the pro-$p$ completion $\widehat{H}_p$ of $H$ is the inverse limit of all $p$-finite quotients of $H$. There is a natural map $$H \lr \widehat{H}_p,$$ that is injective precisely when $H$ is residually $p$-finite.

Let $F_0$ be a free abstract group  with free basis $X$. We assume that $X$ is finite. The pro-$p$ completion $F(X)$ of $F_0$ is the free pro-$p$ group with a basis $X$. It satisfies the universal property that defines abstract free groups but now in the category of pro-$p$ groups i.e. if $G$ is a pro-$p$ group and  $\alpha \colon X \to G$ is a set theoretic map, then there is a unique homomorphism of pro-$p$ groups $$f_{\alpha} : F(X) \lr G$$ that extends $\alpha$. Recall that a homomorphism of pro-$p$ groups by definition is a continuous group homomorphism.

For example $\mathbb{Z}_p$ is a free pro-$p$ group with basis with one element, as it is the pro-$p$ completion of $\mathbb{Z}$.
More generally there is a  definition of free pro-$\mathcal{C}$ group with an infinite basis $X$, where $X$ is a profinite space i.e. an inverse limit of finite sets, see \cite[Chapter ~3]{R-Z}.

Let $G$ be a (topologically) finitely generated pro-$p$ group. Then $G$ is a pro-$p$ quotient of the free pro-$p$ group $F(X)$, with some finite $X$ i.e. there is an epimorphism of pro-$p$ groups $$\pi \colon F(X) \to G.$$ We say that $G$ is finitely presented (as a pro-$p$ group or in the category of pro-$p$ groups) if there is a finite subset $R \subseteq Ker (\pi)$ such that $ Ker (\pi)$ is the smallest closed normal subgroup of $F(X)$ that contains $R$. We write $$G = \langle X \ | \ R \rangle_p.$$

If $H$ is an abstract group that has finite presentation in the category of abstract groups $\langle X \ | \ R \rangle$, then $$\widehat{H}_p = \langle X \ | \ R \rangle_p,$$ in particular the pro-$p$ completion of a finitely presented abstract group is finitely presented (in the pro-$p$ sense). But if the pro-$p$ completion of a finitely generated abstract group $H_0$ is finitely presented (as a pro-$p$ group), this does not imply that $H_0$ is finitely presented as an abstract group, see \cite[Prop. I]{King2} for $p > 2$, $H_0$ finitely generated metabelian.

\subsection{Homology, cohomology and finiteness properties of pro-$p$ groups}
\label{sec:hompro}

Let $G$ be a pro-$p$ group.
By definition the completed group algebra with coefficients in $\mathbb{Z}_p$ is $$\mathbb{Z}_p[[G]] = {\varprojlim} ~\frac{\mathbb{Z}}{p^i \mathbb{Z}} [G/ U],$$ where the inverse limit is over all $i \geq 1$ and $U$ open subgroups of $G$. Similarly we have $$\mathbb{F}_p[[G]] =\mathbb{Z}_p[[G]]/ p \mathbb{Z}_p[[G]] =  {\varprojlim} ~ \mathbb{F}_p[[G/U]]$$
where $\mathbb{F}_p$ is the field with $p$ elements and the inverse limit is again over all open subgroups $U$ of $G$. The completed group algebra plays in the category of pro-$p$ groups the same role as the ordinary group algebra in the category of abstract groups.

The completed group algebras $\mathbb{Z}_p[[G]]$ and $\mathbb{F}_p[[G]]$ are pro-$p$ rings, so they are topological rings. But they have one extremely nice property, they are {\it local} rings i.e. they have a unique maximal ideal, that is the kernel of the obvious augmentation map to $\mathbb{F}_p$ that maps $G$ to $1$.
Because of this there is a pro-$p$  version of the Nakayama lemma  : if $V$ is a pro-$p$ $R[[G]]$-module, where $R = \mathbb{Z}_p$ or $\mathbb{F}_p$, then $V$ is finitely generated if and only if $V/ V \Omega$ is finite, where $\Omega$ is the unique maximal ideal of $R[[G]]$. Note that such a property does not hold in general for modules over abstract groups.

In the category of pro-$p$ modules over pro-$p$ rings we can develop homology and cohomology theory as in the abstract case, see \cite{R-Z}. There are notions of free modules, injective modules, free resolutions, projective resolutions. This allows us to define homology groups $H_i(G, V)$ and cohomology groups $H^i(G, W)$, where $V$ is a pro-$p$  $\mathbb{Z}_p[[G]]$-module and $W$ is a discrete $\mathbb{Z}_p[[G]]$-module.

By \cite[Thm. ~7.7.4]{R-Z} for a pro-$p$ group $G$ the following conditions are equivalent:

a) the cohomological $p$-dimension $cd_p(G) \leq 1$;

b) $H^2 (G, \mathbb{F}_p) = 0$;

c) $G$ is a free pro-$p$ group;

d) $G$ is a projective group.

As a corollary it follows that every closed subgroup of a free pro-$p$ group $G$ is free pro-$p$.

Note that for every pro-$p$ group $G$ we have $$H_1(G, \mathbb{F}_p) \simeq G/ \Phi(G).$$ Thus $G$ is (topologically) finitely generated if and only if $dim_{\mathbb{F}_p} H_1(G, \mathbb{F}_p) < \infty$. This is equivalent to $\dim_{\mathbb{F}_p} H^1(G, \mathbb{F}_p) < \infty$.

If $G$ is (topologically) finitely generated then it is finitely presented (in pro-$p$ sense) if and only if $dim_{\mathbb{F}_p} H_2(G, \mathbb{F}_p) < \infty$. This is equivalent to  $dim_{\mathbb{F}_p} H^2(G, \mathbb{F}_p) < \infty$.

It is interesting to note that geometric methods, that so often work nicely for abstract groups, normally do not apply for pro-$p$ groups. But homological and cohomological methods work nicely.

Next, we proceed with the definition of finiteness properties, that parallels that of Definition \ref{def:fin}.

\begin{definition} \label{def:finp }
Let $G$ be a pro-$p$ group and let $n$ be a nonnegative integer. We say that $G$ has type $FP_n$ if $\mathbb{Z}_p$, as a trivial $\mathbb{Z}_p[[ G]]$-module, admits a projective resolution by projective pro-$p$ $\mathbb{Z}_p[[ G]]$-modules with all morphism being continuous, 
$$ \ldots \lr \widetilde{P}_i \lr \widetilde{P}_{i-1} \ldots \lr \widetilde{P}_0 \lr \mathbb{Z}_p \lr 0,$$
with the first $n$ terms finitely generated as pro-$p$  $\mathbb{Z}_p[[ G]]$-module. 
\end{definition}

By the fact that $\mathbb{Z}_p[[ G]]$ is a local ring it follows that the following conditions are equivalent:
\begin{enumerate}[label=\alph*)]
\item  $G$ is of type $FP_n$
\item  all pro-$p$ homology groups $H_i(G, \mathbb{F}_p)$ are finite for $i \leq n$;
\item all pro-$p$ cohomology groups $H^i(G, \mathbb{F}_p)$ are finite for $i \leq n$.
\end{enumerate}

Indeed the equivalence of the first two conditions is the King criterion Theorem \ref{jeremy1} and the equivalence of the last two conditions follows by the Pontryagin duality, see \cite[Prop. 6.3.6]{R-Z}.
Thus $G$ is finitely presented (as a pro-$p$ group) if and only if it is of type $FP_2$. Thus there is no pro-$p$ version of the Bestvina-Brady example mentioned before. So this is another example that pro-$p$ groups behave differently from abstract groups.

\subsection{Golod--Shafarevich inequality for pro-$p$ groups}
\label{sec:gs}

For a pro-$p$ group define $d(G)$ the minimal number of (topological) generators. In \cite{Wilson} Wilson proved the following  pro-$p$ version of the Golod-Shafarevich inequality.

\begin{theorem} \cite{Wilson} \label{wilson1} 
Let $G$ be a pro-$p$ group with a finite pro-$p$ presentation with $n$ generators and $r$ relations and $d = d(G) > 1$. Then either
$$r \geq n + \frac{1}{4} d^2 - d$$ or for each finitely generated discrete dense subgroup $H$ of $G$ there is a closed normal subgroup $K$ of $G$ such that  $HK/ K$ is an infinite,  torsion group.
\end{theorem}

The following result was proved in \cite{Wilson} as a corollary of Theorem \ref{wilson1}.

\begin{theorem} \cite{Wilson} \label{wilson2} Let $G$ be a finitely presented pro-$p$ group and assume that there does not exist a finitely generated discrete  subgroup $H$ of $G$ normalising $K$ and a closed subgroup $K$ of $G$ such that $HK/ K$ is an infinite,  torsion group. Then
\begin{enumerate}[label=\alph*)]
\item  there is a constant $k$ such that for each open subgroup $U$ of $G$ we have $d(U) \leq k [G : U]^{\frac{1}{2}}$;
\item if $N$ is any closed normal subgroup of $G$ such that $G/ N \simeq \mathbb{Z}_p$ then $N$ is (topologically) finitely generated.
\end{enumerate}
\end{theorem}

The assumptions of  Theorem \ref{wilson2}  apply for soluble groups since every finitely generated, torsion, soluble group is finite.
Part b) from Theorem \ref{wilson2} is very striking. Such a result definitely does not hold for abstract groups, as there are many examples of splittings of a discrete group as a descending or ascending HNN extension, where the associated subgroups do not coincide. There is a pro-$p$ version of the HNN construction but in the category of pro-$p$ groups we do not have ascending or descending pro-$p$ HNN extensions with no coinciding associated subgroups, since this cannot be done for finite $p$-groups.

One corollary of Theorem \ref{wilson2} is that for a finitely presented metabelian pro-$p$ group $G$ we have that if $N$ is any closed normal subgroup of $G$ such that $G/ N \simeq \mathbb{Z}_p$ then $N$ is (topologically) finitely generated. But this property does not characterise finite presentability. Finite presentability was classified later by King in \cite{King2} and will be discussed in the next section.

\section{Metabelian groups} \label{sec:ab}

In this section we compare the theory of finitely presented metabelian abstract and pro-$p$ groups.

\subsection{Finitely presented abstract, metabelian groups}
Suppose $H$ is a finitely generated abstract metabelian group. Then there is    a short exact sequence of abstract groups \[ 1 \lr A_0 \lr H \lr Q_0 \lr 1 \] with $A_0$ and $Q_0$ abelian. Here we view $A_0$ as a $\mathbb{Z} Q_0$-module via the $Q_0$-action induced by conjugation and the additive structure of $A_0$ is the underlying group operation restricted to $A_0$.

In \cite{B-S} Bieri and Strebel define an invariant $\Sigma_{A_0}(Q_0)$ that is a subset of the character sphere $$S(Q_0) = Hom(Q_0, \mathbb{R}) \setminus \{ 0 \} /  \sim,$$ where $\chi_1 \sim \chi_2$ if there is a positive real number $r$ such that $\chi_1 = r \chi_2$.  $A_0$ is said to be 2-tame if $$S(Q_0) = \Sigma_{A_0} (Q_0) \cup - \Sigma_{A_0}(Q_0).$$

\begin{theorem} \cite{B-S}
The following conditions are equivalent:
\begin{enumerate}[label=\alph*)]
\item $H$ is finitely presented;
\item $H$ is of type $FP_2$;
\item $A_0$ is 2-tame as $\mathbb{Z} Q_0$-module.
\end{enumerate}
\end{theorem} 

Later on the above invariant was generalised by Bieri-Neumann-Strebel to the invariant $\S^{1}(G)$ described in Section \ref{sec:fin}. Condition c) from the above theorem can be restated as $S(H) = \Sigma^1(H) \cup - \Sigma^1(H)$.

The fact that b) implies c) was proved by a very beautiful geometric argument that uses algebraic topology, in particular covering maps and  van Kampen´s theorem. The fact that c) implies a) is an algebraic calculation, a very non-trivial one. Note that one unexpected corollary of the above result is that finite presentability does not depend on the type of the extension i.e. does not depend on the element of $H^2(Q_0, A_0)$ that describes the group $H$ but only on the structure of $A_0$ as a $\mathbb{Z} Q_0$-module. This is somewhat surprising since  in general i.e. outside the class of metabelian groups, we expect that the finite presentability depends on the type of the extension.

There is a curious link between $\Sigma_{A_0}(Q_0)$ and valuation theory from commutative algebra. This was used in \cite{B-G} to show that $S(Q_0) \setminus \Sigma_{A_0}(Q_0)$ is a rationally defined spherical polyhedron. This  result itself turned out important and quite useful  in the modern area of tropical geometry.

\subsection{Finitely presented, metabelian pro-$p$ groups}

For general finitely presented groups $G$ in \cite{King}  King proved the following criterion.

\begin{theorem} \cite{King}  \label{jeremy1} Suppose that $G$ is a pro-$p$ group,
 $R$ is either $\mathbb{F}_p$ or $\mathbb{Z}_p$ and $N$  a normal pro-$p$ subgroup of $G$ such that $R[[G/N]]$ is left and right Noetherian.
Then $G$ is of type $FP_m$ as a pro-$p$ group if and only if the  homology groups $H_i(N, R)$ are finitely generated as pro-$p$ $R[[G/N]]$-modules for all $i \leq m$, where the $G/N$ action is induced by the conjugation action of $G$ on $N$.
\end{theorem}

As stated before geometric methods do not pass to pro-$p$ groups. Still there is a complete classification of finitely presented metabelian pro-$p$ groups $G$. As a corollary of the above theorem King obtained the following classification

\begin{theorem} \label{jeremy2}
Let  $ 1 \to A \to G \to Q \to 1$ be a short exact sequence of pro-$p$ groups, where $A$ and $Q$ are abelian and $G$ is (topologically) finitely generated. Then the following conditions are equivalent:
\begin{enumerate}[label=\alph*)]
\item $G$ is finitely presented ( as a pro-$p$ group)
\item $H_2(A, \mathbb{Z}_p) \simeq A \widehat{\wedge} A$ is finitely generated as a $\mathbb{Z}_p[[Q]]$-module via the diagonal $Q$-action.
\end{enumerate}
\end{theorem}

We note that in the above theorem the fact that $G$ is (topologically) finitely generated and $Q$ is (topologically) finitely generated abelian, hence (topologically) finitely presented  imply that $A$ is (topologically) finitely generated as a $\mathbb{Z}_p[[Q]]$-module.

In the above theorem $A \widehat{\wedge} A$ is the exterior product in the category of pro-$p$ modules. It is the inverse limit of $(A/U) \wedge (A/ U)$ over all pro-$p$ $\mathbb{Z}_p[[Q]]$-submodules $U$ of $A$ with $A/ U$ $p$-finite. Note that Theorem \ref{jeremy2} is a direct corollary of Theorem \ref{jeremy1} applied for $N = A$. In this case $Q \simeq G/ A$ is a (topologically) finitely generated abelian pro-$p$ group, hence both rings $\mathbb{Z}_p[[Q]]$ and $\mathbb{F}_p[[Q]]$ are Noetherian.

Note that as observed before a pro-$p$  $\mathbb{Z}_p[[Q]]$-module $V$ is (topologically) finitely generated if and only if $V/ V \Omega$ is finite where $\Omega$ is the unique maximal ideal of  $\mathbb{Z}_p[[Q]]$. In particular  condition b) is equivalent to $ A \widehat{\wedge} A / ( A \widehat{\wedge} A) \Omega$ is finite.

It is interesting to note that in \cite{King2} King defines an invariant  that plays in the theory of pro-$p$ metabelian groups the same role as the complement $\Sigma_{A_0}^c (Q_0) = S(Q_0) \setminus \Sigma_{A_0} (Q_0)$ of the first Bieri-Strebel invariant $ \Sigma_{A_0} (Q_0)$. He calls the invariant geometric in the title of the paper  but actually this invariant is everything but geometric. In some sense King´s invariant resembles the link between  $\Sigma_{A_0}^c (Q_0)$   and valuation theory, as suggested by Bieri and Groves in \cite{B-G}, but all this was adapted in the context of pro-$p$ groups. 

One of the corollaries of the main results from \cite{B-S} is that if $H$ is a finitely presented abstract group without free non cyclic subgroups then every metabelian quotient of $H$ is finitely presented. The main part of the proof uses the topological methods mentioned before and thus is not applicable to pro-$p$ groups. The same statement makes sense for pro-$p$ groups but it is still  an open problem, i.e. it is open whether if  $G$ is a finitely presented pro-$p$ group without free non-procyclic pro-$p$ subgroups then the maximal metabelian pro-$p$ quotient of $G$ is finitely presented pro-$p$.

\subsection{The King invariant}

Let $Q$ be a finitely generated, abelian, pro-$p$ group. We denote by
 $\overline{\mathbb{F}}_p$ the algebraic closure of $\mathbb{F}_p$ and by $\overline{\mathbb{F}}_p[[t]]^{\times}$ the multiplicative group of all invertible elements in the power series ring  $\overline{\mathbb{F}}_p [[t]]$. Following King in \cite{King2} we set
$$T(Q) = \{ \chi : Q \lr \overline{\mathbb{F}}_p[[t]]^{\times} \ | \ \chi \hbox{ is a continuous homomorphism} \},$$
here $\overline{\mathbb{F}}_p[[t]]^{\times} $ is considered as a topological group with topology induced by the sequence of ideals $(t) \supseteq (t^2) \supseteq \ldots \supseteq (t^i) \supseteq \ldots $.
For $\chi \in T(Q)$  we define by 
$$\overline{\chi} : \mathbb{Z}_p[[Q]] \lr \overline{\mathbb{F}}_p[[t]]$$ the unique continuous ring homomorphism
that extends the continuous homomorphism $\chi$.

Let $A$ be a finitely generated pro-$p$ $\mathbb{Z}_p[[Q]]$-module. In \cite{King2} King defined
the following invariant
$$ \Delta(A) = \{  \chi \in T(Q) \ | \ ann_{\mathbb{Z}_p[[Q]]}(A) \subseteq Ker (\overline{\chi})  \}. $$

The following result is an alternative  classification of the finitely presented metabelian pro-$p$ groups given by King in \cite{King2}.

\begin{theorem} \cite{King2} \label{KK} Let $1 \to A \to G \to Q \to 1$ be a short exact sequence of  pro-$p$ groups, where $G$ is a (topologically) finitely generated pro-$p$ group and $A$ and $Q$ are abelian. Then $G$ is a finitely presented pro-$p$ group if and only if $\Delta(A) \cap \Delta(A)^{-1} = \{ 1 \}$.
\end{theorem}

\section{Algebraic fibering for pro-$p$ groups} \label{sec:extpro}

The main purpose of this section is to discuss the pro-$p$ versions of some of the results in Section \ref{sec:main}. The proofs can be found in \cite{Koch}.

\begin{theorem} \label{Main1p} \cite{Koch} Let $1 \to K \to G \to \Gamma \to 1$ be a short exact sequence of pro-$p$ groups and $n_0 \geq 1$ be an integer such that 
\begin{enumerate}
\item $G$ and $K$ are of type $FP_{n_0}$, 
\item $\Gamma^{ab}$ is infinite,
\item there is a normal pro-$p$ subgroup $N$ of $K$ such that $G'\cap K \subseteq N$, $K/ N \simeq \mathbb{Z}_p$ and $N$ is of type $FP_{n_0-1}$. 
\end{enumerate}
Then there is a normal pro-$p$ subgroup $M$ of $G$ such that $G/ M \simeq \mathbb{Z}_p$, $M \cap K = N$ and $M$ is of type $FP_{n_0}$. Furthermore if $K$, $G$ and $N$ are of type $FP_{\infty}$ then $M$ can be chosen of type $FP_{\infty}$.
\end{theorem}

We note that as in the case of abstract groups, in the pro-$p$ case condition 1) from Theorem \ref{Main1p} implies that $\Gamma$ is of type $FP_{n_0}$. Actually it suffices that $G$ is of type $FP_{n_0}$ and $K$ is of type $FP_{n_0-1}$.

We say that a pro-$p$ group $G$ is coherent (as a pro-$p$ group or in the category of pro-$p$ groups) if every finitely generated pro-$p$ subgroup of $G$ is finitely presented as a pro-$p$ group. This  generalises to higher dimensions: a pro-$p$ group $G$ is $n$-coherent if any pro-$p$ subgroup of $G$ that is of  type $FP_n$ is  of type $FP_{n+1}$.  In this notation a pro-$p$ group is 1-coherent if and only if  it is coherent (in the category of pro-$p$ groups).

\begin{corollary} \label{cor1p} \cite{Koch} Let $K$, $\Gamma$ and $G = K \rtimes \Gamma$ be pro-$p$ groups and $n_0 \geq 1$ be an integer such that
\begin{enumerate}
\item $\Gamma$ is finitely generated free pro-$p$ but not  procyclic,
\item $K$ is of type  $FP_{n_0}$, 
\item there is a normal pro-$p$ subgroup $N$ of $K$ such that  $G'\cap K \subseteq N$, $K/ N \simeq \mathbb{Z}_p$ and $N$ is of type $FP_{n_0 -1}$ but is not of type $FP_{n_0}$.  
\end{enumerate}
Then there is a normal pro-$p$ subgroup $M$ of $G$ such that $G/ M \simeq \mathbb{Z}_p$, $M \cap K = N$ and $M$ is of type $FP_{n_0}$ but is not of type $FP_{n_0 + 1}$. In particular $G$ is not $n_0$-coherent.
\end{corollary}

The proofs of the above results are not geometric as in the abstract case, since we do not have geometric tools here as $\Sigma$-theory. But we  use homological methods that are more appropriate for the category of pro-$p$
groups, including the Lyndon-Hochschild-Serre spectral sequence. We also  reduce several problems to studying finite generation of some particular pro-$p$ modules and in some cases use King´s $\Delta$-invariant. Though it has no geometric flavour it should be considered as the pro-$p$ version of the complement of the Bieri-Neumann-Strebel $\Sigma^1$ invariant.

 In \cite{Koch} we use the following result that is a pro-$p$ version of a result of Kuchkuck from \cite{K}. It is interesting to note that the map $\theta$ from the statement of Lemma \ref{L1p} is not supposed to be surjective or injective. Lemma  \ref{L1p} is a very surprising result that needs more attention and deserves consideration not only as an auxiliar result used in the proof of Theorem \ref{Main1p}.

\begin{lemma} \label{L1p} \cite{Koch}  Let $n \geq 1$ be a natural number, \[ 1 \lr A \lr  B \lr  C \lr 1 \] a short exact sequence of pro-$p$ groups with  $A$ of type $FP_n$ and $C$ of type $FP_{n+1}$. Assume there is another short exact sequence of pro-$p$ groups \[ 1 \lr A \lr B_0 \lr C_0 \lr 1 \] with $B_0$ of type $FP_{n+1}$ and that there is a homomorphism of pro-$p$ groups $\theta: B_0 \rightarrow B$ such that $\theta|_A = id_A$,  i.e. there is a commutative diagram of homomorphisms of pro-$p$ groups 
\[ \xymatrix{
1\ar[r]&
A \ar[r]\ar[d]_{id_{A}} &
B_0 \ar[d]_{\theta} \ar[r]^{\pi_0} & C_0 \ar@{.>}[d]^{\nu} \ar[r]&1\\
 1\ar[r]& A 
 \ar[r]& B \ar[r]^{\pi}  &
C \ar[r]  &1 } \]
  Then $B$ is of type $FP_{n+1}$.
\end{lemma}

\section{Coherence and incoherence for pro-$p$ groups} \label{sec:coh}

In \cite{KW22} Kropholler and Walsh asked  and studied the question whether for $F_k$ the free abstract group of rank $k$, $$\hbox{ every group }H = F_n \rtimes F_k \hbox{ with }n,k \geq 2 \hbox{ is not coherent.}$$ As we mentioned, this question also arose in \cite{Wi19}. This question and related topics were studied in \cite{K-V-W}, \cite{FV23}. The main point is to look for a subgroup $H_0$ of finite index in $H$, hence   it is  free-by-free, but with a better homological property, namely $$H_0 = K_0 \rtimes \Gamma_0$$ with $K_0$ and $\Gamma_0$ free of finite rank, $d(\Gamma_0) > 1$ and the short exact sequence $$1 \lr K_0 \lr H_0 \lr \Gamma_0 \lr 1 \hbox{ has excessive homology,}$$ which means that $$rk (H^1(H_0, \mathbb{R})) > rk(H^1(\Gamma_0, \mathbb{R})).$$
Unfortunately as shown in \cite{K-V-W} it is not always possible to find such $H_0$.

The situation for pro-$p$ groups is quite different. It makes sense to consider a pro-$p$ group $$G = K \rtimes \Gamma,$$ where $K$ and $\Gamma$ are (topologically) finitely generated, free, non-procyclic  pro-$p$ groups. Then $\Gamma$ acts by conjugation on $K$ and this gives an action of $\Gamma$ on the abelianization of $K$ i.e.  maximal
abelian pro-$p$ quotient of $K$, that is $\mathbb{Z}_p^n$ for appropriate positive integer $n \geq 2$. This gives a continuous homomorphism $$\rho \colon \Gamma \to GL_n(\mathbb{Z}_p).$$ Note that $GL_n(\mathbb{Z}_p)$ is a profinite group, that is virtually pro-$p$. A pro-$p$ group that (continuously) embeds in $GL_n(\mathbb{Z}_p)$ is pro-$p$ analytic and cannot have  non-procyclic free pro-$p$ subgroups (the class of analytic pro-$p$ groups is an important class of pro-$p$ groups, for more details see \cite{DSMS}).
This is the big difference with the abstract case. The group $GL_n(\mathbb{Z})$  has many free non-cyclic subgroups, but  
$GL_n(\mathbb{Z}_p)$ cannot have closed free non-procyclic  pro-$p$ subgroups! In particular $$\rho \hbox{ is not injective.}$$ Then we can take $\Gamma_0$ a non-procyclic, (topologically) finitely generated pro-$p$ subgroup of $Ker(\rho)$ and consider the group $$G_0 = K \rtimes \Gamma_0.$$ For this group we have that $\Gamma_0$ acts trivially  on the abelianization of $K$. We write $G_0'$ for the commutator subgroup of $G_0$ i.e. the closed subgroup (topologically) generated by $\{ [g_1, g_2] \ | \ g_1, g_2 \in G_0 \}$. Then $$G_0' \cap K = K'$$ and using that for any normal pro-$p$ subgroup $N$ of $K$ such that $K/ N \simeq \mathbb{Z}_p$ we have that $N$ is not (topologically) finitely generated, we can apply Corollary \ref{cor1p} i.e. we construct a pro-$p$ subgroup of $G_0$ that is (topologically) finitely generated but not finitely presented (in pro-$p$ sense),. Thus $G_0$ is pro-$p$ incoherent. 

 This argument can be generalised for a bigger class of groups denoted by $\mathcal{L}$. 
 The class  of pro-$p$ groups $\mathcal{L}$ was defined in \cite{K-Z} by Kochloukova and Zalesskii and should be considered as an attempt to generalise limit groups to the pro-$p$ setting. It contains all (topologically) finitely generated free pro-$p$ groups. A profinite version of it  was studied by Zalesskii and Zapata in \cite{Z-Z}.
 A pro-$p$ group from  $\mathcal{L}$ has many properties that resemble an abstract limit group. The class of abstract limit groups (over free groups) was studied independently by Sela and Kharlampovich, Myasnikov. It coincides with the class of all fully residually free abstract groups. Still we do not have a good understanding what the theory of fully residually free pro-$p$ groups should be. The definition of the class  $\mathcal{L}$ was an attempt to use the extension of centralizers construction of Kharlampovich, Myasnikov to study pro-$p$ groups.

\begin{corollary} \label{CorMain*} \cite{Koch}
Let $ 1 \to K \to G \to \Gamma \to 1$ be a short exact sequence of pro-$p$ groups such that 
\begin{enumerate} 
\item $K$ is a non-abelian pro-$p$ group from the class $\mathcal{L}$, 
\item $\Gamma$ contains a non-abelian, free pro-$p$ subgroup.
\end{enumerate}
Then $G$ is incoherent (in the category of pro-$p$ groups). In particular if $K$ is a finitely generated free pro-$p$ group  with $d(K) \geq 2$ then  $G$ is incoherent (in the category of pro-$p$ groups).
\end{corollary}

We point out that not much is known for pro-$p$ coherence. But in a recent paper Snopce and Zalesskii proved that some pro-$p$ right angled Artin groups are coherent. More precisely if $\Gamma$ is a finite simplicial graph, the right angled Artin pro-$p$ group $G_{\Gamma}$ associated to $\Gamma$ is defined by the pro-$p$ presentation $$\langle V(\Gamma) \ | \ g_1 g_2 = g_2 g_1 \hbox {when } g_1, g_2 \hbox{ are connected by an edge  in } \Gamma \rangle_p.$$ This is the pro-$p$ completion of the corresponding abstract right angled Artin group. The graph $\Gamma$ is called chordal if  every circuit of length at least 4 has a chord ( i.e. there is an edge outside of the circuit that links two vertices from the circuit). Thus every circuit of minimal length should have precisely 3 edges.

\begin{proposition} \cite{S-Z}
Let $\Gamma$ be a chordal graph. Then every (topologically) finitely generated pro-$p$ subgroup of $G_{\Gamma}$ is of homological type $FP_{\infty}$. In particular $G_{\Gamma}$ is coherent (in the category of pro-$p$ groups).
\end{proposition}

 By \cite{Serre} a Demushkin pro-$p$ group $G$ is a Poincare duality group of dimension 2. Thus   $H^i(G, \mathbb{F}_p)$ is finite for all $i$, $\dim H^2(G, \mathbb{F}_p) = 1$ and the cup product $$\cup : H^i(G, \mathbb{F}_p) \times H^{2- i}(G, \mathbb{F}_p) \lr H^2(G, \mathbb{F}_p)$$ is a non-degenerated bilinear form for all $i \geq 0$.

There are two invariants associated to a Demushkin pro-$p$
group: the minimal number of (topological) generators $d$ and $q$ that is either
a power of the prime $p$ or $\infty$. More on Demushkin pro-$p$ groups can be found in \cite{W-book}.
 Demushkin pro-$p$ groups were described in terms of presentations  in \cite{Dem1}, \cite{Dem2}, \cite{La}, \cite{Se}. We state below the classification of Demushkin pro-$p$ groups  due to Demushkin, Labute, Serre.
 
\begin{theorem} \cite{Dem1}, \cite{Dem2} \label{d1p} Let $D$ be a Demushkin group with invariants $d, q$ and suppose
that $q \not= 2$. Then $d$ is even and $D$ is isomorphic to $F /R$, where $F$ is a free pro-$p$ group with basis $x_1, \ldots , x_d$ and $R$ is generated as a normal closed subgroup by
$$x_1^q [x_1, x_2] \ldots [x_{d-1}, x_d],$$
where for $q = \infty$ we define $x_1^{\infty} = 1$. Furthermore all groups having such presentations are Demushkin.
\end{theorem}

\begin{theorem} \cite{Se} \label{d2p} Let $D$ be a Demushkin pro-$2$ group with invariants $d, q$ and suppose that $q = 2$ and $d$ is odd. Then $D$ is isomorphic to $F /R$, where $F$ is a free pro-$2$ group with basis $x_1, \ldots , x_d$ and $R$ is generated as a normal closed subgroup by
$$x_1^2 x_2^{2^f} [x_2, x_3] \ldots [x_{d-1}, x_d]$$
for some integer $f \geq 2$ or $\infty$. Furthermore all groups having such presentations
are Demushkin.
\end{theorem} 

\begin{theorem} \cite{La} \label{d3p} Let $D$ be a Demushkin pro-$2$ group with $d$ even and $q = 2$. Then
$D$ is isomorphic to $F /R$, where $F$ is a free pro-$2$ group with basis $x_1, \ldots, x_d$ and $R$ is generated as a normal closed subgroup either by
$$x_1^{2^f +2}
 [x_1, x_2][x_3, x_4] \ldots [x_{d-1}, x_d]$$ for some integer $f \geq 2$ or $f = \infty$,
or by
$$
x_1^2 [x_1, x_2]x_3^{2^f} [x_3, x_4] \ldots [x_{d-1}, x_d]$$ for some integer $f \geq 2$ or $f = \infty$, $d \geq 4$.
Furthermore all groups having such presentations are Demushkin.
\end{theorem}
 
 As a corollary of Corollary \ref{cor1p} we obtain the following result on incoherence.

\begin{corollary} \label{Demup} \cite{Koch}
Let $1 \to K \to G \to \Gamma \to 1$ be a short exact sequence of pro-$p$ groups such that  
\begin{enumerate}
\item $K$ is a non-abelian pro-$p$ RAAG or a non-soluble Demushkin group,
\item $\Gamma$  is a non-abelian pro-$p$ RAAG or a non-soluble Demushkin group.
\end{enumerate}
Then $G$ is incoherent (in the category of pro-$p$ groups).
\end{corollary}

Let $G$ be a finitely generated pro-$p$ group. Define $$Aut_0(G) = Ker (Aut(G) \lr Aut(G/ \Phi(G))),$$ where $\Phi(G)$ is the Frattini subgroup of $G$. Then $Aut_0(G)$ is a pro-$p$ subgroup of $Aut(G)$ of finite index. 
For a (topologically) finitely generated free pro-$p$ group $F$  the topological group $Aut(F)$ was studied by Lubotsky in \cite{L}.

We are interested whether a pro-$p$ version of the Gordon result from \cite{G} that the automorphism group of the free abstract group of rank 2 is incoherent holds.

\begin{corollary} \label{aut} \cite{Koch}
Suppose that $K$ is a finitely generated free pro-$p$ group with $ d(K) \geq 2$. If $Out(K)$ contains  a pro-$p$  free non-procyclic subgroup then 
$Aut_0(K)$ is incoherent (in the category of pro-$p$ groups).
\end{corollary}

As mentioned before by the  results in \cite{B-S} every metabelian quotient of a finitely presented abstract group  that does not contain free non-cyclic abstract subgroups  is itself finitely presented. But we do not know whether a pro-$p$ version of this holds. Using Corollary \ref{aut} and some results of Romankov from  \cite{R0}, \cite{R}  the following result was proved in \cite{Koch}.

\begin{corollary} \label{alternativep} \cite{Koch}
Suppose that $K$ is a finitely generated free pro-$p$ group  with $d(K)  \geq 2$.  Then either 
$Aut_0(K)$ is incoherent (in the category of pro-$p$ groups) or the pro-$p$ version of the Bieri-Strebel result does not hold.
\end{corollary} 

Recently Marco Boggi has told us that the condition that  $Out(K)$ contains  a pro-$p$  free non-procyclic subgroup from Corollary \ref{aut} can be verified in many cases. We include an appendix that treats this condition.

\setcounter{secnumdepth}{-1}
\section{Appendix: A lemma on the outer automorphism group of a free pro-$p$ group \\
by Marco Boggi, Universidade Federal Fluminense, Brazil}
For a given group $G$, we denote by $\wh{G}^{(p)}$ its pro-$p$ completion and by $d(G)$ the cardinality of a minimal set of generators. 
We then have: 

\begin{lemma} Let $F$ be a free group. 
\begin{enumerate}
\item If $d(F)\geq 3$, then $\out(\wh{F}^{(p)})$ contains a free pro-$p$ group of rank $d(F)-1$.
\item If $d(F)= 2$, then $\out(\wh{F}^{(2)})$ contains a free pro-$2$ group of rank $2$.
\end{enumerate}
\end{lemma}

\begin{proof}(1): Let $S=S_{g,n}$ be a closed connected orientable surface of genus $g$ from which $n$ points have been 
removed and let $\Pi=\Pi_{g,n}$ be its fundamental group for a fixed base point 
$P\in S$. So that $\Pi_{0,n}$ is free of rank $n-1$ and $\Pi_{1,1}$ is free of rank $2$. Let us assume that $2g-2+n>0$ and that $n>0$.
Let $\Pi_P:=\pi_1(S\smallsetminus P, a)$, for some base point $a\in S\smallsetminus P$. Note then that $d(\Pi_P)=d(\Pi)+1$.

Let $\Gamma(S)$ be the mapping class group of the surface S and  $\Gamma(S,P)$ be the 
mapping class group of the marked surface $(S,P)$. There is a short exact
sequence (the Birman sequence, cf. \cite[Section 4.2]{FM12}): 
\[ 1 \lr \Pi \lr \Gamma(S,P) \lr \Gamma(S) \lr 1. \]
The representation $\rho_P \colon \Gamma(S,P) \to \aut(\Pi)$  induced by the action (modulo isotopy) of elements of $\Gamma(S,P)$ on the marked 
surface $(S, P)$ is faithful and is equivalent to the restriction of inner automorphisms from $\Gamma(S,P)$ to its normal subgroup $\Pi$.
In particular, this representation identifies the inclusion $\Pi\subset\Gamma(S,P)$ in the Birman sequence with the inclusion $\Inn\,\Pi\subset \aut(\Pi)$.

The representation $\rho_P^\mathrm{out} \colon \Gamma(S,P)\to\out(\Pi_P)$, induced by
the action (modulo isotopy) of elements of $\Gamma(S,P)$ on the surface $S\ssm P$, is also faithful.
This representation then identifies the subgroup $\Pi$ of $\Gamma(S,P)$ with a subgroup of $\out(\Pi_P)$.

This works also in the pro-$p$ setting. The profinite topologies on $\aut(\hP^{(p)})$ and $\out(\hP_P^{(p)})$ induce
profinite topologies on $\Gamma(S,P)$ via the representations $\rho_P$ and $\rho_P^\mathrm{out}$ and the natural
homomorphisms $\aut(\Pi)\to\aut(\hP^{(p)})$ and $\out(\Pi_P)\to\out(\hP_P^{(p)})$. In principle, these two topologies
might be distinct but they are in fact equivalent (cf.\ Corollary~4.7 in \cite{CongTop}) and called the \emph{$p$-congruence topology}
on $\Gamma(S,P)$. The completion of $\Gamma(S,P)$ with respect to this topology is denoted by $\check{\Gamma}^{(p)}(S,P)$ 
and called the \emph{$p$-congruence completion} of $\Gamma(S,P)$.

Since $\wh{\Pi}^{(p)}$ is center free, there is a natural monomorphism $\wh{\Pi}^{(p)}\cong\Inn(\wh{\Pi}^{(p)})\hookra\aut(\wh{\Pi}^{(p)})$.
This implies that the restriction of the $p$-congruence topology on $\Pi$, via the embedding $\Pi\hookra\Gamma(S,P)$, is the full pro-$p$
topology. Hence, there is a natural monomorphism $\wh{\Pi}^{(p)}\hookra\check{\Gamma}^{(p)}(S,P)$ and thus a natural monomorphism:
\[\wh{\Pi}^{(p)}\hookra\out(\wh{\Pi}_P^{(p)}).\]
For $S=S_{0,n}$ and $n\geq 3$, this gives an embedding of a free pro-$p$ group of rank $n-1$ in the outer automorphism group of a
free pro-$p$ group of rank $n$.
\medskip

\noindent
(2): By  \cite[Proposition 2.5 and 4.1]{HypNotes}, the kernel of the natural representation $\Gamma(S_{1,1})\cong SL_2(\Z)\to SL_2(\Z/2)$ is the direct product of a 
free group $D_2$ of rank $2$, generated by the squares of Dehn twists (these correspond to the transvections of $SL_2(\Z)$,
via the above isomorphism), and the center of $\Gamma(S_{1,1})$, which is a cyclic group of order $2$.

The conclusion then follows from the $2$-congruence property for $\Gamma(S_{1,1})$
(cf.\ (i) Theorem~A in \cite{HI} or also (ii) Theorem 1.4 in \cite{CongTop}). This in fact implies that the natural homomorphism 
$\wh{D_2}^{(2)}\to\out(\hat{\Pi}_P^{(2)})$, induced by the representation $\rho_P^\mathrm{out} : \Gamma(S_{1,1})\hookra\out(\Pi_P)$, 
is also faithful.

Note that a similar argument would fail for primes $p\geq 11$, since the $p$-congruence
subgroup property for $\Gamma(S_{1,1})$ does not hold in these cases by (ii) of Theorem~A in \cite{HI}. 
\end{proof}


\end{document}